\documentclass{article}
\usepackage{graphicx} % Required for inserting images
\usepackage{amsmath, amssymb, amsthm}
\usepackage[hidelinks]{hyperref}
\usepackage{url}
\usepackage{authblk}
\usepackage{pgf,tikz}
\usetikzlibrary{arrows}

\newtheorem{theorem}{Theorem}
\newtheorem{conjecture}[theorem]{Conjecture}

\newtheorem{claim}[theorem]{Claim}

\newcommand{\F}{\mathcal{F}}
\newcommand{\R}{\mathbb{R}}
\renewcommand{\S}{\mathbb{S}}
\newcommand{\proj}{\operatorname{proj}}
\newcommand{\rot}{\operatorname{rot}}

\newcommand{\Sd}{\mathbb{S}^{d}}
\newcommand{\Red}{\mathbb{R}^d}
\newcommand{\vol}{\operatorname{vol}}

\newcommand{\conv}{\operatorname{conv}}

\title{$k$-dimensional transversals for %$d$-dimensional
fat convex sets\thanks{Both authors were supported by the ERC Advanced Grant ``ERMiD'' and by the Thematic Excellence Program TKP2021-NKTA-62 of the National Research, Development and Innovation Office. The first author was also supported by the NKFIH grants FK132060 and SNN135643. The second author was also supported by the J\'anos Bolyai Research Scholarship of the Hungarian Academy of Sciences and by \'UNKP-23-5 of NRDIO.}}
\author[1,2]{Attila Jung}
\author[1,2]{Dömötör Pálvölgyi} %domotor.palvolgyi@ttk.elte.hu
\affil[1]{ELTE Eötvös Loránd University, Budapest, Hungary}
\affil[2]{HUN-REN Alfréd Rényi Institute of Mathematics, Budapest, Hungary}
\begin{document}

\maketitle

\begin{abstract}
    We prove a fractional Helly theorem for $k$-flats intersecting fat convex sets. A family $\mathcal{F}$ of sets is said to be $\rho$-fat if every set in the family contains a ball and is contained in a ball such that the ratio of the radii of these balls is bounded by $\rho$. We prove that for every dimension $d$ and positive reals $\rho$ and $\alpha$ there exists a positive $\beta=\beta(d,\rho, \alpha)$ such that if $\mathcal{F}$ is a finite family of $\rho$-fat convex sets in $\mathbb{R}^d$ and an $\alpha$-fraction of the $(k+2)$-size subfamilies from $\mathcal{F}$ can be hit by a $k$-flat, then there is a $k$-flat that intersects at least a $\beta$-fraction of the sets of $\mathcal{F}$. We prove spherical and colorful variants of the above results and prove a $(p,k+2)$-theorem for $k$-flats intersecting balls.
\end{abstract}

\section{Introduction}\label{sec:intro}

We say that a family $\mathcal{T}$ of geometric objects \emph{hits} or \emph{pierces} another family $\mathcal{F}$ if for every $F \in \mathcal{F}$ there is a $T \in \mathcal{T}$ such that $T \cap F \neq \emptyset$. If $\mathcal{T}$ consists of a single $k$-flat ($k$-dimensional affine subspace), then we say that $\mathcal{F}$ has a $k$-transversal. In this language, Helly's theorem \cite{helly1923mengen} states, that for a finite family $\mathcal{F}$ of convex sets in $\mathbb{R}^d$, if every subfamily of size $d+1$ has a $0$-transversal, i.e., a point contained in all $d+1$ sets, then the whole family has a $0$-transversal, i.e., a point contained in all the sets. Vicensini~\cite{vincensini1935figures} asked whether one can generalize Helly's theorem to $k$-transversals for $0 < k < d$, but this was shown to be false by Santal{\'o}~\cite{santalo1940theorem}, who constructed arbitrarily large families of convex sets in the plane without a $1$-transversal whose every proper subfamily has a $1$-transversal.

A natural next question in this direction is whether the following fractional Helly theorem of Katchalski and Liu has any analog for $k$-transversals.

\begin{theorem}[Katchalski and Liu \cite{katchalski1979problem}]\label{thm:fh}
    For every dimension $d$ there exists a function $\beta\colon (0,1] \to (0,1]$ with the following property.
    Let $\mathcal{F}$ be a finite family of convex sets in $\mathbb{R}^d$ and $\alpha \in (0,1]$. If at least $\alpha\binom{|\mathcal{F}|}{d+1}$ of the $(d+1)$-size subfamilies of $\mathcal{F}$ have a $0$-transversal, then there exists a subfamily of $\mathcal{F}$ of size at least $\beta(\alpha) |\mathcal{F}|$ which has a $0$-transversal.
\end{theorem}

Alon and Kalai showed that it can indeed be generalized to $(d-1)$-transversals (hyperplanes) intersecting convex sets.

\begin{theorem}[Alon and Kalai \cite{alon1995bounding}]\label{thm:FHforHyperplanes}
    For every dimension $d$ there exists a function $\beta\colon (0,1] \to (0,1]$ with the following property. Let $\mathcal{F}$ be a finite family of convex sets from $\mathbb{R}^d$ with $|\mathcal{F}|=n$.
    If for some positive $\alpha$ at least $\alpha\binom{n}{d+1}$ of the $(d+1)$-size subfamilies of $\mathcal{F}$ have a $(d-1)$-transversal, then there exists a subfamily of $\mathcal{F}$ of size at least $\beta(\alpha)n$ which has a $(d-1)$-transversal.
\end{theorem}

Later, Alon, Kalai, Matou{\v s}ek and Meshulam showed that the above statement does not generalize to $1$-flats (lines) intersecting convex sets in $\mathbb{R}^3$.

\begin{theorem}[Alon, Kalai, Matou{\v s}ek and Meshulam \cite{alon2002transversal}]
    For every integers $n$ and $m$ there exists a family of at least $n$ convex sets in $\mathbb{R}^3$ such that any $m$ of them have a $1$-transversal, but no $m+4$ of them have a $1$-transversal.
\end{theorem}

By a suitable lifting argument this implies that there are no fractional Helly theorems for $k$-flats intersecting convex sets if $0 < k < d-1$. Following this example it is natural to look for fractional Helly-type theorems for special kinds of sets and $k$-flats with $0<k<d-1$. For example, Matou{\v s}ek proved, that $k$-flats intersecting a family of sets with bounded description complexity admit a fractional Helly theorem~\cite{matousek2004bounded}. A family of sets is said to have bounded description complexity, if each of the sets can be described using a bounded number of bounded degree polynomial inequalities.
\begin{theorem}[Matou{\v s}ek~\cite{matousek2004bounded}]
    For every dimension $d$ there exists a function $\beta\colon (0,1] \to (0,1)$ with the following property.
    Let $\mathcal{F}$ be a finite family of sets of bounded description complexity in $\mathbb{R}^d$ and $\alpha \in (0,1]$. If at least $\alpha\binom{|\mathcal{F}|}{(k+1)(d-k)+1}$ of the $((k+1)(d-k)+1)$-size subfamilies of $\mathcal{F}$ have a $k$-transversal, then there exists a subfamily of $\mathcal{F}$ of size at least $\beta(\alpha) |\mathcal{F}|$ which has a $k$-transversal.
\end{theorem}

We prove a fractional Helly theorem for fat convex sets. A family $\mathcal{F}$ of sets is $\rho$-fat, if for every $K \in \mathcal{F}$ there exist $x_K \in \mathbb{R}^d$ and $r_K,R_K \in \mathbb{R}$ such that $R_K \leq \rho r_K$ and $K$ contains a ball of radius $r_K$ centered at $x_k$ and is contained in a ball of radius $R_K$ centered at $x_k$.\footnote{Note that requiring these balls to be concentric is not a significant restriction; if $K$ is contained in some ball of radius $R$, then it is contained in a ball of radius $2R$ centered at any point of $K$.}

\begin{theorem}\label{thm:FHballs}
    For every $d$ and $\rho$ there exists a function $\beta\colon (0,1] \to (0,1)$ with the following property.
    Let $\mathcal{F}$ be a finite family of $\rho$-fat convex sets in $\mathbb{R}^d$, $0\le k\le d$ an integer, and $\alpha \in (0,1]$.
    If at least $\alpha \binom{|\mathcal{F}|}{k+2}$ of the $(k+2)$-tuples of $\mathcal{F}$ have a $k$-transversal, then $\mathcal{F}$ has a subfamily of size at least $\beta(\alpha)|\mathcal{F}|$ which has a $k$-transversal.
\end{theorem}

As discussed above, the $k=d-1$ case is known to hold even without the fatness assumption \cite{alon1995bounding}, but the %$k=1, d\geq3$ 
$0<k<d-1$ case is known to be false without fatness, even if we replace the $(k+2)$-tuples in the assumption with $m$-tuples for any finite $m$ \cite{alon2002transversal}.

Note that as any convex body in $\R^d$ can be made $d$-fat with an affine transformation by John's theorem \cite{john1948}, our result also holds for the family of homothets of any fixed convex body. 
For these, Matou{\v s}ek's result would only work if the convex body has bounded union complexity, like a ball or a polyhedron, and even in this case would only give $((k+1)(d-k)+1)$-size subfamilies, which is improved\footnote{This is indeed an improvement because by averaging if some $\alpha$ fraction of the $((k+1)(d-k)+1)$-tuples intersect, then also some $\alpha$ fraction of the $(k+2)$-tuples intersect.} by our $(k+2)$-size subfamilies for all $k<d-1$. (This is an improvement because by averaging if some $\alpha$ fraction of the $((k+1)(d-k)+1)$-tuples intersect, then also some $\alpha$ fraction of the $(k+2)$-tuples intersect.)
%In an even more special case, we improve the $(k+1)(d-k)+1$ bound of Matou{\v s}ek to $k+2$ in the case of $k$-flats intersecting balls in $\mathbb{R}^d$.

The proof of our main Theorem~\ref{thm:FHballs} can be found in Section \ref{sec:fh}. Along the way we prove spherical and colorful variants of the result as well.\\

Now we turn to the history of our other main result.
A celebrated generalization of Helly's theorem is the following $(p,q)$-theorem of Alon and Kleitman.

\begin{theorem}[Alon and Kleitman \cite{alon1992piercing}]\label{thm:pq}
For every $p \geq d+1$ there exists a $C$ such that the following holds.
If $\mathcal{F}$ is a family of compact convex sets in $\mathbb{R}^d$ such that among any $p$ members of $\mathcal{F}$ there are $d+1$ with a $0$-transversal, then there exist at most $C$ points hitting all of $\mathcal{F}$.
\end{theorem}

The fractional Helly theorem of Katchalski and Liu plays a crucial role in the proof of Theorem~\ref{thm:pq}. It would be very interesting to know whether our fractional Helly theorem for $k$-flats intersecting $\rho$-fat convex sets has such a corollary.

\begin{conjecture}
    For every $p,q,d$ and $\rho$ there exists a $C$ such that the following holds. If $\mathcal{F}$ is a family of $\rho$-fat compact convex sets in $\mathbb{R}^d$ such that among any $p$ members of $\mathcal{F}$ there are $k+2$ with a $k$-transversal, then there exist at most $C$ $k$-flats hitting all of $\mathcal{F}$.
\end{conjecture}

We do not even know whether the statement holds if we replace $k+2$ by any $m>k+2$ which does not depend on $\mathcal{F}$. The missing ingredient is a weak $\varepsilon$-net theorem for $k$-flats intersecting $\rho$-fat convex bodies. In Section~\ref{sec:pq}, we discuss how two of the main approaches used to prove the existence of (weak) $\varepsilon$-nets fails in our case. As one of the standard approaches can be used if we replace $\rho$-fat convex sets with balls ($1$-fat convex sets), we obtain the following corollary. 

\begin{theorem}\label{thm:pqballs}
    For every three positive integers $d$, $k < d$ and $p \geq k+2$ there exists a $C=C(d,k,p)$ such that the following holds. If $\mathcal{F}$ is a (possibly infinite) family of closed balls in $\mathbb{R}^d$ such that among any $p$ members of $\mathcal{F}$ there are $k+2$ that can be hit by a single $k$-flat, then there exist at most $C$ $k$-flats hitting all of $\mathcal{F}$.
\end{theorem}

Theorem~\ref{thm:pqballs} follows from results in Section \ref{sec:pq}. We show spherical and colorful variants of this result as well. Note that a weaker version of Theorem~\ref{thm:pqballs}, where we replace $k+2$ by $(k+1)(d-k)+1$ follows from the more general results of Matou{\v s}ek~\cite{matousek2004bounded}.

Keller and Perles showed that if $\mathcal{F}$ is an infinite family of $d$-dimensional balls, and among any $\aleph_0$ balls some $k+2$ balls have a $k$-transversal, then the whole family can be pierced with finitely many $k$-dimensional affine subspaces \cite{keller2022aleph_0}. In a forthcoming paper \cite{jung2024infinite} it is shown that this infinite variant follows from our Theorems~\ref{thm:FHballs} and \ref{thm:pqballs} with a purely combinatorial proof. 

Our Theorem~\ref{thm:pqballs} generalizes to certain uniformly fat families as follows. A family $\mathcal{F}$ of sets in $\mathbb{R}^d$ is called $(r,R)$-fat if for every $K \in \mathcal{F}$ there is an $x_K \in \mathbb{R}^d$ such that $B(x_K, r) \subseteq K \subseteq B(x_K, R)$. Note that an $(r,R)$-fat family is $\rho$-fat with $\rho = \frac{R}{r}$ but a $\rho$-fat family might not be $(r,R)$-fat for any fixed $r,R$ as it can contain arbitrarily small and large sets.
Because of this, the argument described in the next paragraph, does not apply directly to $\rho$-fat sets in $\mathbb{R}^d$.

Fractional Helly and $(p,q)$-type results about $k$-transversals of (congruent) balls generalize easily to (not necessarily convex) $(r,R)$-fat sets as were described in \cite{ghosh2022heterochromatic,keller2022aleph_0}. We give a short summary of the argument here. If among every $p$ members of an $(r,R)$-fat family $\mathcal{F}$, there are $k+2$ which have a $k$-transversal, then so do the $(k+2)$-tuples of balls of radius $R$ containing them. We can apply Theorem~\ref{thm:pqballs} to the family of $R$-balls to obtain a bounded size family $\mathcal{T}$ of $k$-flats hitting all the balls. By considering parallel $k$-flats close to members of $\mathcal{T}$, we can conclude that bounded many of them stabs the $r$-balls contained in the members of $\mathcal{F}$.
Therefore, there is a bounded size family of $k$-flats hitting every member of $\mathcal{F}$. This argument works even in the case when the members of $\mathcal{F}$ are not convex. Not even connectedness is required, only $(r,R)$-fatness.

Holmsen and Matou\v sek asked whether a fractional Helly number exists for disjoint translates of a convex set in $\R^3$ with respect to lines \cite{holmsen2004}.
As any convex body in $\mathbb{R}^3$ can be made $3$-fat with an affine transformation, our Theorem~\ref{thm:FHballs} answers their question. But we cannot really take credit for this, as a positive answer already follows from a result of Matou\v sek~\cite{matousek2004bounded} combined with the trick of making the set $(1,3)$-fat with an affine transformation, and then using the reduction to balls that we sketched earlier.
This method would only give that the fractional Helly number is at most $5$ (without using our Theorem \ref{thm:FHballs}); it was observed by Dobbins and Holmsen a few years ago (personal communication), independently of us, that this can be improved to the optimal $3$, and they have also obtained fractional Helly theorems for $(r,R)$-fat sets with respect to $k$-flats (unpublished).

As another direction generalizing Helly-type theorems to $k$-transversals, we mention a result of Hadwiger which shows that although Helly's theorem does not generalize directly to $1$-transversals, a finite family of pairwise disjoint convex sets in the plane has a $1$-transversal if and only if the family has a linear ordering such that any $3$-tuple of convex sets has a $1$-transversal which is consistent with the ordering \cite{hadwiger1957} (see the definition there). Hadwiger's result has been generalized to $(d-1)$-transversals by Goodman, Pollack and Wenger \cite{goodman1993geometric}, and very recently to $k$-transversals for any $0 \leq k < d$ by McGinnis and Sadovek~\cite{mcginnis2024necessary}.

\section{Fractional Helly Theorems}\label{sec:fh}

In this section we prove Theorem \ref{thm:FHballs} in the below more general, colorful form, which also generalizes the colorful version of the fractional Helly theorem due to B{\'a}r{\'a}ny, Fodor, Montejano, Oliveros, and P{\'o}r \cite{barany2014colourful}. 

\begin{theorem}[B{\'a}r{\'a}ny et al. \cite{barany2014colourful}]\label{thm:CFHballs}
    For every $\rho\geq1$ and every dimension $d$ there exists a function $\beta: (0,1] \to (0,1)$ with the following property.    
    Let $\mathcal{F}_1, \ldots, \mathcal{F}_{k+2}$ be families of $\rho$-fat convex sets in $\mathbb{R}^d$. If at least $\alpha |\mathcal{F}_1|\cdots |\mathcal{F}_{k+2}|$ of the colorful selections have a $k$-transversal, then there exists an $i$ with $\mathcal{F}_i$ having a subfamily of size at least $\beta(\alpha)|\mathcal{F}_i|$ which has a $k$-transversal.
\end{theorem}

Our inductive proof of Theorem~\ref{thm:CFHballs} also requires a spherical analog, for which we need the following definitions. A cap of the sphere $\S^d$ is the intersection of $\S^d$ with a ball in $\mathbb{R}^{d+1}$ (we can assume that the center of the ball is from $\S^d$), a great $k$-sphere is the intersection of $\S^d$ with a $(k+1)$-dimensional linear subspace of $\mathbb{R}^{d+1}$ (assuming that the origin is the center of $\S^d$), and a spherical $k$-transversal for a spherical family is a great $k$-sphere intersecting all members of the family. We can  describe caps as subsets of $\S^d$ of the form $B(x,\varepsilon) = \{y \in \S^d: \angle(x,y) = \cos^{-1}(xy) \leq \varepsilon\}$ where $x \in \S^d$. %Enélkül nem lenne jó a rho-fatség defje!
Beware that the definition of $B(x,r)$ implicitly depends on the space we are working in!
We intentionally do not use a different notation so that our argument can be described more generally, but it will be always clear from the context in which space we are in.
A $\rho$-fat family of spherical sets can be defined analogously to the Euclidean case, with the existence of caps $B(x_K, r_K)$ and  $B(x_K,R_K)$ for every $K \in \mathcal{F}$ with the property that  $B(x_K, r_K) \subset K \subset B(x_K,R_K)$ and $R_K \leq \rho r_K$.

\begin{theorem}\label{thm:CFHcaps}
    For every $\rho \geq 1$ and every dimension $d$ there exists a function $\beta: (0,1] \to (0,1)$ with the following property.    
    Let $\mathcal{F}_1, \ldots, \mathcal{F}_{k+2}$ be families of $\rho$-fat convex sets in $\S^d$. If at least $\alpha |\mathcal{F}_1|\cdots |\mathcal{F}_{k+2}|$ of the colorful selections have a spherical $k$-transversal, then there exists an $i$ with $\mathcal{F}_i$ having a subfamily of size at least $\beta(\alpha)|\mathcal{F}_i|$ which has a spherical $k$-transversal.
\end{theorem}

Notice that in fact Theorem~\ref{thm:CFHcaps} implies Theorem~\ref{thm:CFHballs} as any counterexample in $\R^d$ would also give a counterexample on the surface of a large enough sphere $\S^d$ with small (compared to the radius of the sphere) $\rho$-fat sets on it.
However, we prove both Theorems~\ref{thm:CFHballs} and \ref{thm:CFHcaps} as we believe that our argument is easier to understand in the more natural setting of $\R^d$, and then think it through that it also works in $\S^d$ without much change.
%We recommend the reader to first picture the argument in the more natural setting of $\R^d$, and then think it through that it also works in $\S^d$ without much change.
We mark the differences in our simultaneous proof for both theorems with brackets [~].

The following is a key property of $\rho$-fat convex sets.

\begin{claim}\label{cl:fatproperty}
    For all $d$ and $\rho$ we have a $\gamma > 0$ such that the following holds. If $K$ is a $\rho$-fat convex set in $\Red$ [or in $\Sd$], $y \in K$ and $t \leq R_K$, then $\vol(K \cap B(y,t)) \geq \gamma\vol(B(y,t))$, where $B(y,t)$ denotes the ball [or spherical cap] of radius $t$ centered at $y$. (See Figure~\ref{fig:fat-prop}.)
\end{claim}

\begin{figure}
    \centering
\begin{tikzpicture}[line cap=round,line join=round,x=1.0cm,y=1.0cm]
\clip(7.08,2.46) rectangle (13.16,7.75);
\draw(10.2,5.1) circle (0.39cm);
\draw(10.2,5.1) circle (2.6cm);
\draw(8.1,5.1) circle (0.72cm);
\draw (8.1,5.1)-- (10.1,5.48);
\draw (8.1,5.1)-- (10.09,4.72);
\draw (8.1,5.1)-- (8.5,4.5);
\draw (10.2,5.1)-- (10.35,4.74);
\draw (10.2,5.1)-- (12.4,3.72);
\draw (9.25,5.34) node[anchor=north west] {C};
\draw (7.7, 5.1) -- (9.1, 6.6) -- (12.1, 5.1) -- (9.1, 3.6) -- cycle;
\begin{scriptsize}
\fill [color=black] (10.2,5.1) circle (1.5pt);
\draw[color=black] (10.32,5.22) node {$x_K$};
\fill [color=black] (8.1,5.1) circle (1.5pt);
\draw[color=black] (8.18,5.32) node {$y$};
\draw[color=black] (8.45,4.75) node {$t$};
\draw[color=black] (10.12,4.90) node {$r_K$};
\draw[color=black] (11.54,3.86) node {$R_K$};
\draw[color=black] (9.75,6.00) node {$K$};
\end{scriptsize}
\end{tikzpicture}
\caption{A property of fat convex sets}
    \label{fig:fat-prop}
\end{figure}    

\begin{proof}
    Let $B(x_K, r_K) \subset K \subset B(x_K, R_K)$ with $R_K \leq \rho r_K$, and define $r_K'=\min(r_K,t)\ge t/\rho$ and $C = \operatorname{conv}(\{y\} \cup B(x_K, r_K)) \subset K$.
    $C$ contains all the points $u$ of $B(y,r_K')$ such that $\angle (u-y, x_K-y)$ is at most some number depending only on $\rho$, so $C$ contains a constant fraction of $B(y,t)$.
    % Let $B(x_K, r_K) \subset K \subset B(x_K, R_K)$ with $R_K \leq \rho r_K$. If $x_K \in B(y,t)$, then already $B(x_K, r_K)$ contains a constant fraction of $B(y, t)$ as $t\le R_k\le tr_K$.
    % %$r_K \geq R_k/\rho\ge t/\rho$.
    % Otherwise, $C = \operatorname{conv}(\{y\} \cup B(x_K, r_K)) \subset K$ contains a constant fraction of $B(y,t)$, as $C$ contains all the points $u$ of $B(y,t)$ such that $\angle (u-y, x_K-y)$ is at most some number depending only on %$d$ and
    % $\rho$.
\end{proof}

\begin{proof}[Proof of Theorems~\ref{thm:CFHballs} and \ref{thm:CFHcaps}]
The proof is by induction on $k$ and $d$. % in Section~\ref{sec:CFHproofs}.
Let $n_i = |\mathcal{F}_i|$. Let $\mathcal{H}$ be the $(k+2)$-uniform hypergraph with vertex set $\mathcal{F} = \cup \mathcal{F}_i$, where the edges are those colorful $(k+2)$-tuples which have a [spherical] $k$-transversal. Charge every hyperedge to its member $K$ with the smallest bounding radius $R_K$. In case of ties, charge arbitrarily to one of the members $K$ with smallest $R_K$.
    
Since sets are charged $\alpha n_1 \cdots n_{k+2}$ times, there is a color class $\mathcal{F}_i$ whose members are charged $\frac{\alpha}{k+2} n_1 \cdots n_{k+2}$ times. Without loss of generality, we may assume that this color class is $\mathcal{F}_{k+2}$. Pick the set $K_0$ from $\mathcal{F}_{k+2}$ with the most charge. By averaging, $K_0$ was charged at least $\frac{\alpha}{k+2}n_1\cdots n_{k+1}$ times.
    
%We prove by induction on $k$ and $d$.
If $k=0$, then we have $\frac\alpha2 n_1$ sets from $\mathcal{F}_1$ intersecting $K_0\in \mathcal{F}_2$. We have a $y_K \in B(x_{K_0}, R_{K_0}) \cap K$ for all such $K\in \mathcal{F}_1$. As $K$ contains a positive fraction of the ball $B(y_K, R_{K_0})$ by Claim~\ref{cl:fatproperty}, and $B(y_K, R_{K_0})$ occupies a positive fraction of $B(x_{K_0}, 2R_{K_0})$, we have that every $K$ contains a positive fraction of $B(x_{K_0}, 2R_{K_0})$. In this case, a constant fraction of the $\frac\alpha2 n_1$ sets can be hit by a single point by the volumetric pigeonhole principle (where the constant depends only on the dimension and $\rho$).

%Euclidean and the spherical $(k,d)$-cases to the spherical $(k-1,d-1)$-case
If $k > 0$, we can reduce finding a [spherical] $k$-transversal in $\R^d$ [in $\S^d$] to finding a spherical $(k-1)$-transversal in $\S^{d-1}$ as follows.
\smallskip

\textbf{Euclidean case:} If we are in the Euclidean space $\mathbb{R}^d$, we may assume that $B_0 = B(x_{K_0}, R_{K_0})$ is the unit ball centered at the origin. Let $K' = K + B_0$ (where $+$ denote the Minkowski sum) if $B \in\cup \mathcal{F}_i\setminus\{K_0\}$, and let $K_0'$ be the degenerate ball containing only the origin. Denote the new families by $\mathcal{F}' = \cup \mathcal{F}_i'$.

\begin{claim}
    If $k+2$ sets from $\mathcal{F}$ have a $k$-transversal, then the corresponding $k+2$ balls in $\mathcal{F}'$ have a $k$-transversal.
\end{claim}
\begin{proof}
    If none of the sets is $K_0$, the same $k$-transversal works.
    Otherwise, translate the hitting $k$-flat to the origin.
    As the translation is by at most $R_{K_0}$, if any $K$ was hit by the original $k$-flat, then $K'$ is hit by the translated $k$-flat.
\end{proof}

Now, centrally project every set in $\mathcal{F}'\setminus \{K_0'\}$ to the surface of the unit sphere from the origin. (If a set contains the origin, its projection is the entire sphere.)

\begin{claim}
    A projection of a $\rho$-fat convex set is $\frac\pi2\rho$-fat.
\end{claim}
\begin{proof}
    The projection of a ball $B(x,r)$ to the unit sphere centered at the origin is a cap of the unit sphere with radius $\sin^{-1}(r/t)$ if $x$ is at distance $t \geq r$ of the origin. Thus, the projection of a $\rho$-convex set $K$ is $\frac{\sin^{-1}(R_K/t)}{\sin^{-1}(r_K/t)}$-fat if $x_K$ is at distance $t > R_K$ from the origin. But $t \leq \sin^{-1} (t) \leq \frac{\pi}{2}t$, thus $\frac{\sin^{-1}(R_K/t)}{\sin^{-1}(r_K/t)} \leq \frac{\pi}{2}\frac{ R_k}{ r_k }= \frac\pi2 \rho$ and the image of $K$ is $\frac\pi2\rho$-fat.

    If $t \leq R_K$, then the projection of $B(x_K, r_K)$ is a cap with radius $\sin^{-1}(r_K/t) \geq \frac{r_K}{t}\geq \frac{1}{ \rho} $. Either the projection of $K$ is contained in a cap of radius $\pi/2$ and contains a cap of radius $\frac{1}{\rho}$, or the image of $K$ is the complete unit sphere. In both cases the projection is $\frac\pi2 \rho$-fat. 
\end{proof}
Denote the obtained family of $\rho' = \frac\pi2\rho $-fat spherically convex sets by $\mathcal{F}^*$. As $\frac{\alpha}{k+2}n_1\cdots n_{k+1}$ of the colorful selections of $\mathcal{F}_1', \ldots, \mathcal{F}_{k+1}'$ have a $k$-dimensional \emph{linear} subspace hitting them, we have that $\frac{\alpha}{k+2}n_1\cdots n_{k+1}$ of the colorful selections of $\mathcal{F}_1^*, \ldots, \mathcal{F}_{k+1}^*$ have a great $(k-1)$-sphere intersecting them.
The assumptions of the spherical colorful fractional Helly Theorem~\ref{thm:CFHcaps} on $\S^{d-1}$ with great $(k-1)$-spheres are satisfied, thus we have an $i \in [k+1]$ and a subfamily $\mathcal{G}^* \subset \mathcal{F}_i^*$ of size at least $\beta^* |\mathcal{F}_i^*|$ which can be hit by a great $(k-1)$-sphere, where $\beta^*$ is the value we obtain from Theorem~\ref{thm:CFHcaps} with parameters $d-1,k-1$ and $\frac{\alpha}{k+2}$.
This means that the corresponding family of preimages $\mathcal{G}' = \{K' \in \mathcal{F}': \proj_{\mathbb{S}^{d-1}} K' \in \mathcal{G}^*\}$ can be hit by a $k$-dimensional linear subspace $F$.
Let $\mathcal{G}=\{K\in \mathcal F: K'\in \mathcal{G}'\}$.
If we project $K_0$ and every set $K\in \mathcal{G}$ to the orthogonal complement $F^\perp$ of $F$, every projected set $\proj_{F^\perp} K$ will intersect the set $\proj_{F^\perp} K_0$.

Since for all $K$ there exists a $y_K \in B(0,1) \cap \proj_{F^\perp} K$, and $\proj_{F^\perp} K$ contains a ball of $F^\perp$ of radius $\rho$, we know that $K$ occupies a constant fraction of $B(y_K, 1) \cap F^\perp \subset B(0,2) \cap F^\perp$ by Claim~\ref{cl:fatproperty} and thus a positive fraction of $B(0,2)\cap F^\perp$. By the pigeonhole principle we can find a point $x\in F^\perp$ in the intersection of a constant fraction of the projections $\{\proj_{F^\perp} K: K \in \mathcal{G}'\}$.
$F+x$ (the translation of $F$ by $x$) is a $k$-flat intersecting a constant fraction of the fat convex sets in $\mathcal{G}'$. 
\smallskip

\textbf{Spherical case:}
For this case, we need some claims about the geometry of spherical caps and great $k$-spheres which are described in Section~\ref{sec:prelim}, but otherwise the main argument is the same as in the Euclidean case.
First, we deal with the case where there are a lot of large sets in $\cup_i\mathcal{F}_i$. Note that we may assume that all the $n_i$s are large enough, otherwise a $\beta$-fraction of a color class will consists of a single set $K$ (if $\beta$ is chosen to be small enough). Let $c < \alpha/(k+2)$ be a small constant. If there is an $\mathcal{F}_i$ such that for at least $cn_i$ of the sets $K \in \mathcal{F}_i$ we have $R_K > \pi/8$, then a constant fraction of them can be hit with a single point by the volumetric pigeonhole principle, as all of them contain caps with radius at least $\rho \pi/8$. This constant fraction depends only on $d$ and $\rho$. If we do not have this many large sets in any of the $\mathcal{F}_i$, we can simply delete the large ones, as this way we loose at most $c(k+2)n_1\ldots n_{k+2}$ out of the original $\alpha n_1 \ldots n_{k+2}$ colorful $(k+2)$-tuples with a spherical $k$-transversal. Thus, we can have families of size at least $(1-c)n_i$ with at least an $(\alpha- c(k+2))n_1 \ldots n_{k+2}$ colorful $(k+2)$-tuples having a spherical $k$-transversal and no set $K$ in the families with $R_K > \pi/8$.
%Now assume that $R_K < \frac{\pi}{4}$ for all $K \in \cup_i \mathcal{F}_i$.

Let $K_0$ be the most charged set as defined above, let $\varepsilon = R_{K_0}$ and $B_0 = B(x_{K_0},\varepsilon)$. For a set $K \subset \S^d$, its $\varepsilon$-neighborhood is $K^\varepsilon = \{u \in \S^d: \exists v \in K \text{ with } \angle(u,v) \leq \varepsilon\}$, and $\operatorname{conv}(K)$ is the intersection of all the spherically convex sets containing $K$. If $K$ is a $\rho$-fat convex set, then $\conv(K^\varepsilon)$ is a $\rho$-fat convex set by Claim~\ref{cl:sphericalNeighborhood}. For every $K \in \cup \mathcal{F}_i \setminus \{K_0\}$, let $K' = \conv(K^\varepsilon)$ and let $K_0' = \{x_{K_0}\}$. Denote the new families of spherically convex $\rho$-fat sets by $\mathcal{F}' = \cup \mathcal{F}_i'$. By Claim~\ref{cl:hitLargerCaps}, if some sets in $\mathcal{F}$ can be hit by a great $k$-sphere, then the corresponding sets in $\mathcal{F}'$ can be hit by a great $k$-sphere as well. Now let $\mathcal{F}^* = \proj_{B_0} \mathcal{F}'$, where $\proj_{B_0}$ denotes the projection into the boundary $\partial B_0$. By Claim~\ref{cl:capProj}, $\mathcal{F}^*$ is a family of spherically convex $\rho' = \frac{\pi^3}{4}\rho^2$-fat sets of the $(d-1)$-dimensional sphere $\partial B_0$. Without loss of generality, assume that $K_0 \in \F_{k+2}$. As $K_0$ is the most charged set, at least a $\frac{\alpha}{k+2}$-fraction of the colorful selections of $\F_1', \ldots, \F_{k+1}'$ have a spherical $k$-transversal containing $x_0$. These great $k$-spheres of $\S^d$ become great $(k-1)$-spheres of $\partial B_0$ after the projection. Thus, at least a $\frac{\alpha}{k+2}$-fraction of the colorful selections of $\F_1^*, \ldots, \F_{k+1}^*$ have a $(k-1)$-transversal. The assumptions of the spherical colorful fractional Helly Theorem~\ref{thm:CFHcaps} on $\partial B_0$ with spherical $(k-1)$-transversals are satisfied, thus we have an $i$ and a subfamily $\mathcal{G}^* \subset \mathcal{F}_i^*$ of size at least $\beta^* |\mathcal{F}_i^*|$ which has a spherical $(k-1)$-transversal. This means that the corresponding family of preimage sets $\mathcal{G}'$ has a spherical $k$-transversal. By Claim~\ref{cl:hitSmallerCaps}, a large subfamily of $\mathcal{G}$ has a spherical $k$-transversal as well.
\end{proof}

\section{Caps and projections on the sphere}\label{sec:prelim}

In this section, we prove the claims about spherical sets that were used in the proof of Theorem \ref{thm:CFHcaps}.
All the arguments are fairly straight-forward modifications of the simple proofs used in the Euclidean case, but more tedious calculations are required.

Recall that $B(v, \varepsilon)$ is the spherical cap centered at $v$ with angle $\varepsilon$, and a great $k$-sphere of $\S^d$ is the intersection of $\S^d$ with a $(k+1)$-dimensional linear subspace of $\mathbb{R}^{d+1}$ \cite{aronov2002helly}. A great $1$-sphere is called a great circle. A set $K \subset \S^d$ is convex, if either $K=\S^d$, or for no $v\in\S^d$ we have both $v,-v \in K$, and $K$ contains the shorter great circle arc connecting any two points from $K$.

If $u \notin \{v,-v\}$, then let $\rot(u,v)$ be the rotation $r$ with $r(u) = v$ which is constant on $\operatorname{span}(\{u,v\})^\perp$, so we rotate with the $uov$ angle $\angle (u,v)$ around the center $o$, keeping $\operatorname{span}(\{u,v\})^\perp$ fixed.
If $u \notin \{v,-v\}$, the projection of $u \in \S^d$ onto the boundary $\partial B(v, \varepsilon)$ of the cap $B(v, \varepsilon)$ is the intersection of $\partial B(v, \varepsilon)$ with the halfplane embedded in $\mathbb{R}^{d+1}$ containing $o$ and $v$ on its boundary, and $u$ in its interior.
In other words, if $u \not\in B(v, \varepsilon)$, then the projection is the point of $\partial B(v, \varepsilon)$ that is hit when we rotate $u$ to $v$.
It is denoted by $\proj_{B} (u)$, or simply $\proj (u)$ if $B$ is clear from the context. For any set $X$ and cap $B = B(v,\varepsilon)$, let $\proj_B X=\{\proj_B (u) : u\in X\}$ if $X \cap \{v,-v\}=\emptyset$, and let $\proj_B X = B$ if $X \cap \{v, -v\} \neq \emptyset$.

\begin{claim}\label{cl:sphericalNeighborhood}
  If $K \subseteq \mathbb{S}^d$ is $\rho$-fat, then $\conv(K^\varepsilon)$ is $\rho$-fat as well.
\end{claim}

\begin{proof}
    If $\conv(K^\varepsilon) = \mathbb{S}^d$, then it is $\rho$-fat with $\rho = 1$. Otherwise, as $B(x_K, r_K + \varepsilon) \subseteq \conv(K^\varepsilon) \subseteq B(x_K, R_K + \varepsilon)$, the $\rho$-fatness of $K^\varepsilon$ follows from $\frac{R_K + \varepsilon}{r_K + \varepsilon} \leq \frac{R_K}{r_K} \leq \rho$.
\end{proof}

\begin{claim}\label{cl:hitLargerCaps}
    If the spherical sets $K_1, \ldots, K_n \subset \Sd$  can be pierced with a great $k$-sphere, and $K_1\subset B(v, \varepsilon)$, then $B(v,0), K_2^\varepsilon, \ldots, K_n^\varepsilon$ can also be pierced with a great $k$-sphere.
\end{claim}

\begin{proof}
    Take a great $k$-sphere $F$ intersecting all of $K_1, \ldots, K_n \subset \Sd$, let $v' = \proj_F(v)$  and $\rot = \rot(v', v)$.
    For any point $u \in \S^d$, the spherical distance between $u$ and $\rot(u)$ is at most $\varepsilon$.
    Hence, if $u \in K_i \cap F$, then $\rot(u) \in K_i^\varepsilon \cap \rot(F)$, thus the great $k$-sphere $\rot(F)$ intersects all of $B(v,0), K_2^\varepsilon, \ldots, K_n^\varepsilon$.
\end{proof}

We want to show that during a gnomonic projection the distance of two points close to the center of the projected image cannot be distorted too much.
This is probably a well-known statement, but we could not find it anywhere, so we include the simple calculation below. If we embed $\S^d$ into $\R^{d+1}$ as the unit sphere, the spherical distance of $x$ and $y$ becomes $\angle(x,y)$, while their Euclidean distance is $|x-y|$.

\begin{claim}
    Embed $\S^d$ into $\R^{d+1}$ as the unit sphere. If $u,v \in B(w,\pi/4) \subset \mathbb{S}^d$ and $p$ is the central projection from the origin to the supporting hyperplane of $\mathbb{S}^d$ at $w$, then we have $\angle(u,v) \leq |p(u)-p(v)| \leq 2 \angle(u,v)$.
\end{claim}

\begin{proof}
    As $u,v,w$ are contained in a $2$-dimensional subsphere, it is enough to show the claim for $d=2$. First we calculate the distortion of length if $u$ and $w$ have the same latitude of longitude.
    
    Case 1: $u,v,w$ are on the same great circle. As we have $|p(u)-w| = \tan (\angle(u,w))$ and $1\leq \tan'(\varphi) \leq 2$ if $\varphi \in (0,\pi/4)$, we have $\angle(u,v) \leq |u-v| \leq 2 \angle(u,v)$.

    Case 2: $\angle(u,w) = \angle(v,w) = \varphi$. In this case $p$ acts on them as a homothety with ratio $1/\cos(\varphi)$, which is between $1$ and $\sqrt{2}$ if $\varphi \in (0,\pi/4)$.

    We can argue that these are the extreme cases and the distortion of length is always between $1$ and $2$. We show that for close enough points the distortion of distance is between $1$ and $2$. But then the global distortion has to be in the same interval.
    
    For every $u \in B(w,\pi/4)$ the points of $\S^d$ with a small enough given distance from $u$ form a circle inside $B(w,\pi/4)$. Its image under $p$ is an ellipse. Due to symmetry the axes of the ellipse correspond to point pairs with the same latitude or longitude. We have seen in Cases 1 and 2 that the ratios of the lengths of the axes and the radius of the circle are between $1$ and $2$. But then for every point $v$ in the circle we have $\angle(u,v) \leq |p(u)-p(v)| \leq 2 \angle(u,v)$.
    %Version B: Take two close enough points $u,v$ of $B(w,\pi/4)$ with $\angle(u,w) \geq \angle(v,w)$. Imagine we compute the projection of $u$ and $v$ in two steps. First, we take a homothety with the origin as center which maps $v$ to $p(v)$ and $u$ to $q(u) = |p(v)|u$. As this homothety has ratio $1/\cos(\angle(v,w)) \leq \sqrt{2}$, we have $d_E(q(u), p(v)) \leq \sqrt{2}\angle(u,v)$.
    %TODO: $q(u)$-rol $p(u)$-ra atteres se torzitja a tavolsagot jobban, mint $\sqrt{2}$
    %Thus, locally, the distance of close points increases at most by a factor of $\sqrt{2}\sqrt{2} = 2$. But then the same is true for any two points of $B(w,\pi/4)$.
\end{proof}

\begin{claim}\label{cl:distortion}
    If $K \subseteq B(w,\pi/8) \subset \mathbb{S}^d$ is convex and $\varepsilon \leq \pi/8$, then $\conv(K^\varepsilon) \subset K^{2\varepsilon}$.
\end{claim}

\begin{proof}
    If $u \in \conv (K^\varepsilon)$, then there are $u', u'' \in K^\varepsilon$ such that $u \in \conv(\{u', u''\})$. Let $v', v'' \in K$ be such that $\angle(u',v'), \angle(u'', v'') \leq \varepsilon$. Let $p$ be the central projection of $B(w,\pi/4)$ from the origin to the supporting hyperplane of $\mathbb{S}^d$ at $w$. As $u',u'', v', v'' \in B(w, \pi/4)$, we have $|p(u')- p(v')|, |p(u'')- p(v'')| \leq 2\varepsilon$ by Claim~\ref{cl:distortion}. As the projections of $\conv(\{u', u''\})$ and $\conv(\{v', v''\})$ are segments of a Euclidean space whose endpoints have distance at most $2 \varepsilon$, there is a point $v \in \conv(\{v', v''\})$ such that $|p(v)- p(u)| \leq 2 \varepsilon$. We have $\angle(u,v) \leq 2 \varepsilon$ by Claim~\ref{cl:distortion}. As $v \in K$, this proves $u \in K^{2\varepsilon}$.
\end{proof}

\begin{claim}\label{cl:capProj}
    Let $B = B(v, \varepsilon)$ be a cap, and $K$ be a $\rho$-fat convex set of $\Sd$. The projection $\proj_B K$ is a $\frac{\pi^3}{4}\rho^2$-fat convex set of the $(d-1)$-dimensional sphere $\partial B$.
\end{claim}

\begin{proof}
    If $\proj_B K = B$, then it is convex by definition. Otherwise, as $\proj_B$ maps (shorter) great circle arcs to (shorter) great circle arcs or points, the projection is convex. Let $\rho' = \frac{\pi^3}{4}\rho^2$. To show $\rho'$-fatness, we need to examine how the radius of $B(x_K, r_K)\subset K$ and $B(x_K, R_K)\supset K$ change after the projection.
    
    Assume that $B(x, r)$ is any cap of $\S^d$ and $v \not\in B(x,r)$. Embed $\S^d$ in $\mathbb{R}^{d+1}$ as the unit sphere and look at the simplex spanned by $\underline{0}, x, v$ and $b$ where $b \in \partial B(x, r)$. The radius $\varphi$ of $\proj_{B(v,\varepsilon)} B(x,r)$ will be $\max_b \angle(b-v,x-v)$.
    
    As $\frac{\sin \angle(b-v,x-v)}{|x-b|} = \frac{\sin \angle (x-b, v-b)}{|x-v|}$ by the law of sines, $\angle(b-v,x-v)$ is maximal if $\angle(x-b,v-b) = \pi/2$. Fix such a $b$. We have $\sin \varphi = \frac{|x-b|}{|x-v|}$.

    Let $\angle(v, x) = t$. As $|x| = |v| = |b| = 1$, we have $|x-v| = 2 - 2\cos t$ and $|x-b| = 2 - 2\cos r$ by the law of cosines. As $1 - \cos t = \sin^2(t/2)$ and $ \sin t \leq t$, we have $\varphi \geq \sin \varphi = \frac{|x-b|}{|x-v|} =  \frac{\sin^2(r/2)}{\sin^2(t/2)}$. If $\varphi \leq \pi/2$, then we also have $\varphi \leq \frac{\pi}{2}\sin \varphi = \frac{\pi}{2}\frac{\sin^2(r/2)}{\sin^2(t/2)}$.

    Now returning to the proof, for simplicity we will write $\proj$ for $\proj_B$.
    Let $\proj B(x_K, r_K) = B(\proj x_K, \varphi_r)$ and $\proj B(x_K, R_K) = B(\proj x_K, \varphi_R)$.

    If $\varphi_R \geq \pi/2$, then $R \geq t$. In this case we have either $\varphi_r \geq \pi/2$ or $\varphi_r \geq \sin \varphi_r = \frac{\sin^2(r/2)}{\sin^2(t/2)} \geq \frac{4}{\pi^2}\frac{r^2}{t^2} \geq \frac{4}{\pi^2}\frac{r^2}{t^2} \geq \frac{4}{\pi^2\rho^2}$ and so $\varphi_R \leq \pi \leq \frac{\pi^3}{4}\rho^2\varphi_r$.
    
    Otherwise $\varphi_R < \varphi/2$ and we have $\frac{\varphi_R}{\varphi_r} \leq \frac{\pi}{2}\frac{\sin \varphi_R}{\sin \varphi_r} = \frac{\pi}{2}\frac{\sin^2(R/2)}{\sin^2(r/2)} \leq \frac{\pi^3}{8}\frac{R^2}{r^2} \leq \frac{\pi^3}{8}\rho^2$. As $\proj B(x_K, r_K) \subset \proj K \subset \proj B(x_K, R_K)$, in this case $K$ is $\frac{\pi^3}{8}\rho^2$-fat.
\end{proof}

\begin{claim}\label{cl:hitSmallerCaps}
    For every $d$, $\rho$ and $k$ there exists a constant $c>0$ such that the following holds for every $\varepsilon > 0$. Let $K_2, \ldots, K_n$ be $\rho$-fat convex sets of $\Sd$ with $\varepsilon \leq R_{K_i}$ for all $i \in \{2, \ldots, n\}$. If $\conv(K_2^\varepsilon), \ldots, \conv(K_n^\varepsilon)$ can be hit with a great $k$-sphere, then $cn$ sets from $K_2, \ldots, K_n$ can be hit with a great $k$-sphere as well.
\end{claim}

\begin{proof}
    If $R_{K_i} \geq \pi/8$ for at least half of the sets, then, as those sets all contain caps of radius at least $\pi/(8\rho)$, a constant fraction of them can be hit even with a single point by the volumetric pigeonhole principle. Otherwise delete all the sets $K_i$ with $R_{K_i} \geq \pi/8$ and at least half of the sets remain. Let the family of the remaining at least $n/2$ sets be $\mathcal{H}$.
    
    Let $F$ be the great $k$-sphere hitting all the sets of $\{\conv (K^\varepsilon): K \in \mathcal{H}\}$. In this case,l $K \in \mathcal{H}$ are at distance at most $2\varepsilon$ from $F$ by Claim~\ref{cl:distortion}.
    
    If $k=0$, let $F = \{\pm v\}$. For all $K \in \mathcal{H}$ there exists a $y_{K} \in K$ which is at distance at most $2\varepsilon$ from one of the two antipodal points of $F$. As $K$ occupies a constant fraction of $B(y_{K}, \varepsilon)$ by Claim~\ref{cl:fatproperty} and $B(y_{K}, \varepsilon)$ contains a constant fraction of $B(v, 2\varepsilon) \cup B(-v, 2\varepsilon)$, the volumetric pigeonhole principle yields a point intersecting $cn$ of the sets of $\mathcal{H}$.

    We can reduce the $k>0$ case to the $k=0$ case with a suitable projection. For a $K \in \mathcal{H}$, let $f(K)$ be a closest point of $F$ to $K$. Let $\delta = \delta(d,k) > 0$ be a small enough number, and let $N$ be a metric $\delta$-net of $F$ (its size depends only on $k$ and $d$). By the pigeonhole principle, there exists a point $u\in N$ and a subfamily $\mathcal{H}' \subset \mathcal{H}$ with $|\mathcal{H}'| \geq c'|\mathcal{H}|$ such that for all $K \in \mathcal{H}'$, the point $f(K)$ has distance at most $\delta$ from $u$. 

    Embed $\S^d$ into $\mathbb{R}^{d+1}$ as the unit sphere and let $V$ be the $(k+1)$-dimensional linear subspace of $\mathbb{R}^{d+1}$ with $V \cap \S^d = F$. Let $u, u_1, \ldots, u_{k}$ be an orthonormal basis of $V$. Let $B_i = B(u_i, \pi/2)$ be the halfsphere centered at $u_i$, let $\proj_i = \proj_{B_i}$ be the projection from $\S^d$ to $\partial B_i$ as defined above, and let $\proj_{F,u} = \proj_1 \circ \proj_2 \circ \ldots \circ \proj_k$. The repeated application of Claim~\ref{cl:capProj} gives us that if $K \subset \S^d$ is a $\rho$-fat convex set, then $\proj_{F_u} K$ is a $\rho'=c^k\rho^{2k}$-fat convex set where $c \geq 1$ is a universal constant.
    
    Apply $\proj_{F,u}$ to $F$ and the sets of $\mathcal{H}'$ to get projections in the great $(d-k)$-sphere $\partial B_2 \cap \ldots \cap \partial B_k$. This way $F$ becomes a pair of antipodal points $\{\pm u\}$, the projections of members of $\mathcal{H}'$ are $\rho'$-fat, and if $\delta(d,k)$ is sufficiently small, they are at distance at most $3\varepsilon$ from $u$. By the $k=0$ case we can find a point intersecting a constant fraction of the sets in $\{\proj_{F,u} K: K \in \mathcal{H}'\}$. The preimage of this point under $\proj_{F,u}$ is part of a great $k$-sphere of $\mathbb{S}^d$, and hits a constant fraction of $\mathcal{H}'$, thus a constant fraction of $\mathcal{H}$ and of $\{K_2, \ldots, K_n\}$.
\end{proof}

\section{\texorpdfstring{$(p,k+2)$}{(p,k+2)}-theorems}\label{sec:pq}

The specific $\rho=1$ case of the fractional Helly results of Section \ref{sec:fh}  implies $(p,q)$-type results for $k$-flats intersecting balls. Theorem~\ref{thm:CFHballs} implies the following generalization of Theorem \ref{thm:pqballs}.

\begin{theorem}\label{thm:Cpqballs}
    For every three positive integers $d$, $k < d$ and $p \geq k+2$ there exists a $C$ such that the following holds. If $\mathcal{F}_1,\ldots,\mathcal{F}_p$ are families of closed balls in $\mathbb{R}^d$ such that for every $B_1 \in \mathcal{F}_1,\ldots,B_p\in\mathcal{F}_p$  there are $k+2$ balls that can be hit by a single $k$-flat, then there exist an $i\in [p]$ and at most $C$ $k$-flats hitting all balls of $\mathcal{F}_i$.
\end{theorem}

The following is a corollary of Theorem~\ref{thm:CFHcaps}.

\begin{theorem}\label{thm:Cpqcaps}
    For every three positive integers $d$, $k < d$ and $p \geq k+2$ there exists a $C$ such that the following holds. If $\mathcal{F}_1,\ldots,\mathcal{F}_p$ are finite families of closed caps in $\S^d$ such that for every $B_1 \in \mathcal{F}_1,\ldots,B_p\in\mathcal{F}_p$  there are $k+2$ caps that can be hit by a great $k$-sphere, then there exist an $i\in [p]$ and at most $C$ great $k$-spheres hitting all caps of $\mathcal{F}_i$.
\end{theorem}

Note that the case where the families of balls (caps) can be infinite follows from the finite case by a standard compactness argument. The proofs of the finite versions are simple applications of the Alon-Kleitman method established in \cite{alon1992piercing}, later described in a more general setting in \cite{alon2002transversal}, and adapted to colorful variants in \cite{barany2014colourful}. The abstract (purely combinatorial) proof described in \cite{alon2002transversal} has two main steps. To state them, we phrase the fractional Helly and $(p,q)$-theorems in an abstract, purely combinatorial language as follows. To keep the presentation simple, we only describe the arguments in a non-colorful setting. The interested reader can find a version of the colorful arguments in \cite{barany2014colourful}.

Let $\mathcal{H}$ be a set system over the base set $V$. For an integer $q$ and a function $\beta\colon (0,1] \to (0,1]$, the set system $\mathcal{H}$ satisfies the fractional Helly property FH($q, \beta$) if the following holds. For every $\alpha > 0$ and every finite family $\mathcal{G} \subset \mathcal{H}$, if $\alpha\binom{|\mathcal{G}|}{q}$ of the $q$-tuples of $\mathcal{G}$ has a nonempty intersection, then there exists a subfamily $\mathcal{G}' \subset \mathcal{G}$ of size $\beta(\alpha)|\mathcal{G}|$ such that all the members of $\mathcal{G}'$ have a point in common. The  Katchalski and Liu fractional Helly Theorem \ref{thm:fh} states that if $V = \mathbb{R}^d$, and $\mathcal{H}$ consists of all the convex sets, then there exists a $\beta$ such that $\mathcal{H}$ satisfies FH($d+1, \beta$). Our Theorem~\ref{thm:FHballs} states that if $V = \{k\text{-flats of }\mathbb{R}^d\}$ and $\mathcal{H}$ consists of families of $k$-flats intersecting a common $\rho$-fat convex set, then $\mathcal{H}$ satisfies FH($k+2, \beta$) with some $\beta$.

A set system $\mathcal{H}$ satisfies the $(p,q)$-condition if among any $p$ members of $\mathcal{H}$, some $q$ intersect. Note that for $p \geq q \geq r$, the $(p,q)$-condition implies the $(p,r)$-condition. The transversal number $\tau(\mathcal{H})$ denotes the minimum size of a set $T \subset V$ with $T \cap H \neq \emptyset$ for all $H \in \mathcal{H}$. The fractional transversal number $\tau^*(\mathcal{H})$ is the minimum $\sum_{v \in V} t(v)$ over functions $t\colon V \to [0,1]$ with $\sum_{v \in H} t(v) \geq 1$ for all $H \in \mathcal{H}$. Alon, Kalai, Matou\v sek and Meshulam proved the following two theorems. 

\begin{theorem}[Theorem 8 in Alon, Kalai, Matou\v sek and Meshulam \cite{alon2002transversal}]
If $\mathcal{H}$ satisfies FH($q, \beta$) and the $(p,q)$-condition with some $p \geq q$, then $\tau^*(\mathcal{H}) \leq O_{p,q, \beta}(1)$.
\end{theorem}

Let $\mathcal{H}^\cap$ be the family of all sets that are the intersections of some members of $\mathcal{H}$.

\begin{theorem}[Theorem 9 in Alon, Kalai, Matou\v sek and Meshulam \cite{alon2002transversal}]\label{thm:weakeps}
For every $q$ and $\beta$ there exists a function $f_{q,\beta}$ such that if $\mathcal{H}^\cap$ satisfies FH($q, \beta$), then $\tau(\mathcal{H}) \leq f_{q,\beta}(\tau^*(\mathcal{H}))$.
\end{theorem}

Note that if $\mathcal{H}^\cap$ satisfies $F(q, \beta)$ and $\mathcal{H}$ satisfies the $(p,q)$ condition, then its transversal number is bounded by a function of $p,q$ and $\beta$ as a consequence of the above two theorems. As the family of convex sets of $\mathbb{R}^d$ is closed under intersection, and satisfies $FH(d+1,\beta)$, this provides an alternative proof of Theorem~\ref{thm:pq} of Alon and Kleitman. But the same reasoning is usually not applicable to questions about $k$-transversals with $k > 0$, as the corresponding $\mathcal{H}^\cap \neq \mathcal{H}$ is hard to analyze, for example, for balls the intersections are no longer fat.
However, in this special case, another proof method works.

If we restrict the question to $k$-transversals of balls [caps], as in Theorem~\ref{thm:Cpqballs} [Theorem~\ref{thm:Cpqcaps}], then we can bound $\tau$ with a function of $\tau^*$ and thus prove our $(p,k+2)$-theorems. They are consequences of a classical result of Haussler and Welzl~\cite{haussler1986epsilon}
. The VC-dimension of $\mathcal{H}$ is the supremum of sizes of subsets $S \subseteq V$ with the property that for every $X \subseteq S$ there is an $H \in \mathcal{H}$ with $H \cap S = X$.

\begin{theorem}[Haussler and Welzl \cite{haussler1986epsilon}]\label{thm:eps}
    For every $d$ there exists a function $f_d$ such that if $\mathcal{H}$ has VC-dimension at most $d$, then $\tau(\mathcal{H}) \leq f_d(\tau^*(\mathcal{H}))$.
\end{theorem}

The VC-dimension of $k$-flats intersecting balls [caps] is bounded by the Milnor-Thom theorem on sign patterns of polynomials \cite{milnor1964betti, thom1965homologie}, because a family of $k$-flats intersecting a given ball [cap] can be described by a bounded number of polynomial inequalities of bounded degree.

As for $k$-flats intersecting $\rho$-fat convex sets neither Theorem~\ref{thm:weakeps}, nor Theorem~\ref{thm:eps} is applicable, the missing ingredient is proving a $(p,q)$-theorem for $k$-flats intersecting $\rho$-fat convex sets.

\begin{conjecture}
    There is a function $f$ such that if $\mathcal{H}$ consists of families of $k$-flats intersecting members of a family of $\rho$-fat convex sets in $\mathbb{R}^d$, then $\tau(\mathcal{H}) \leq f(\tau^*(\mathcal{H}))$.
\end{conjecture}

\section*{Acknowledgement}

We would like to thank Andreas Holmsen for useful discussions, and for bringing \cite{matousek2004bounded} to our attention.

\bibliographystyle{plain}
\bibliography{biblio}

\end{document}